\documentclass[a4paper]{article}
\usepackage{fullpage}
\usepackage{tikz}

\usepackage[pdftex,breaklinks,colorlinks,
linkcolor=blue,
citecolor=blue,
urlcolor=blue]{hyperref}

\usepackage{enumitem}
\usepackage{amsmath,amssymb,amsfonts,amsthm}

\newtheorem{theorem}{Theorem}

\newtheorem{lemma}[theorem]{Lemma}
\newtheorem{proposition}[theorem]{Proposition}
\newtheorem{corollary}[theorem]{Corollary}

\theoremstyle{definition}
\newtheorem{definition}[theorem]{Definition}
\newtheorem{remark}[theorem]{Remark}

\author{Pablo Romero\footnote{Facultad de Ingenier\'ia, Universidad de la Rep\'ublica, Montevideo, Uruguay. E-mail address: \texttt{promero@fing.edu.uy}}
\footnote{Facultad de Ciencias Exactas y Naturales, Universidad de Buenos Aires, Ciudad Universitaria. Av. Int. Güiraldes 2160. Buenos Aires, Argentina.}}

\date{}

\makeatletter%
\begin{document}

\title{Existence of most reliable two-terminal graphs\\ with distance constraints}

\maketitle

\begin{abstract}\let\thefootnote\relax
A two-terminal graph is a simple graph equipped with two distinguished vertices, called terminals. Let $T_{n,m}$ be the class consisting of all nonisomorphic two-terminal graphs on $n$ vertices and $m$ edges. Let $G$ be any two-terminal graph in $T_{n,m}$, and let $d$ be any positive integer. For each $\rho\in [0,1]$, the \emph{$d$-constrained two-terminal reliability of $G$ at $\rho$}, denoted $R_G^d(\rho)$, is the probability that $G$ has some path of length at most $d$ joining its terminals after each of its edges is independently deleted with probability $\rho$. We say $G$ is a \emph{$d$-uniformly most reliable two-terminal graph} ($d$-UMRTTG) if for each $H$ in $T_{n,m}$ and every $\rho \in [0,1]$ it holds that $R_{G}^d(\rho)\geq R_H^d(\rho)$. Previous works studied the existence of $d$-UMRTTG in $T_{n,m}$ when $d$ is greater than or equal to $n-1$, or equivalently, when the distance constraint is dropped. In this work, a characterization of all $1$-UMRTTGs and $2$-UMRTTGs is given. Then, it is proved that there exists a unique $3$-UMRTTG in $T_{n,m}$ when $n\geq 6$ and $5 \leq m \leq 2n-3$. Finally, for each $d\geq 4$ and each $n\geq 11$ it is proved that there is no $d$-UMRTTG in $T_{n,m}$ when $20 \leq m \leq 3n-9$ or when $3n-5 \leq m \leq \binom{n}{2}-2$.
\end{abstract}

\renewcommand{\labelitemi}{--}

\section{Introduction} \label{intro}
Let $T_{n,m}$ be the set consisting of all nonisomorphic two-terminal graphs on $n$ vertices and $m$ edges. 
Let $G$ be in $T_{n,m}$ and let $d$ be any positive integer. A \emph{$d$-pathset of $G$} is a spanning subgraph of $G$ 
having some path of length at most $d$ joining its terminals. For each $\rho$ in $[0,1]$, the \emph{$d$-constrained two-terminal reliability of $G$ at $\rho$}, denoted $R_G^d(\rho)$, is the probability of $G$ being a $d$-pathset after each of its edges is independently deleted with probability $\rho$. Given positive integers $d$, $n$, and $m$ such that $n\geq 2$ and $1\leq m \leq \binom{n}{2}$, the question we want to answer is whether there exists $G$ in $T_{n,m}$ such that $R_G^d(\rho) \geq R_H^d(\rho)$ for each $H$ in $T_{n,m}$ and every $\rho$ in $[0,1]$. Such a two-terminal graph $G$ is called a \emph{$d$-uniformly most reliable two-terminal graph} ($d$-UMRTTG).  

In this work we discuss the existence of $d$-UMRTTGs in each nonempty set $T_{n,m}$  
as a function of the distance $d$. This document is organized as follows. Section~\ref{section:concepts} presents basic concepts. The body of related work is covered in Section~\ref{section:background}. A characterization of all  
$1$-UMRTTGs and $2$-UMRTTGs in each nonempty set $T_{n,m}$ is given in Section~\ref{section:1and2}. 
In Section~\ref{section:3} it is proved that there exists a unique $3$-UMRTTG in $T_{n,m}$ when $n\geq 6$ and $5\leq m \leq 2n-3$. In strong contrast with the previous cases, in Section~\ref{section:contrast} it is proved that for each $d\geq 4$ and each $n\geq 11$ there is no $d$-UMRTTG in $T_{n,m}$ when $20 \leq m \leq 3n-9$ or when $3n-5 \leq m \leq \binom{n}{2}-2$.

\section{Concepts}\label{section:concepts}
Let $A$ be any finite set. We denote $|A|$ the number of elements in $A$. For each nonnegative integer $i$, 
we denote $\binom{A}{i}$ the set consisting of each of the sets composed by $i$ elements in $A$. 
All graphs in this work are finite, simple, and undirected. Let $G$ be any graph. We denote the vertex set and the edge set of $G$ by $V(G)$ and $E(G)$, respectively. Let $u$ and $v$ be vertices in $V(G)$. We say 
$u$ and $v$ are \emph{adjacent} if $uv$ is in $E(G)$. The \emph{open neighborhood of $v$}, denoted $N_G(v)$, 
is the set consisting of all vertices in $G$ that are adjacent to $v$. The \emph{closed neighborhood of $v$}, denoted $N_G[v]$, is $N_G(v)\cup \{v\}$. The \emph{degree of $v$} is $|N_G(v)|$. We say $v$ is a \emph{hanging vertex} if it has degree is $1$. 
We say $v$ is a \emph{universal vertex} if $N_G[v]=V(G)$. The vertices $u$ and $v$ are \emph{true twins} when $N_G[u]=N_G[v]$. Observe that true twins are adjacent. The graph $G$ is \emph{regular} if each of its vertices has the same degree. The graph $G$ is \emph{almost-regular} if the degrees of each pair of its vertices differ by at most one. For each vertex set $X$ in $V(G)$, 
$uX$ denotes the set consisting of all edges $uw$ in $G$ such that $w$ is in $X$. We denote the $n$-path and the $n$-complete graph by $P_n$, and $K_n$, respectively. The \emph{length} of $P_n$ is $n-1$. The \emph{distance} between $u$ and $v$ is the length of the shortest path in $G$ whose endpoints are $u$ and $v$. A \emph{subgraph of $G$} is a graph $H$ such that $V(H)\subseteq V(G)$ and $E(H)\subseteq V(G)$. If $G$ is not $K_1$ then the \emph{edge-connectivity} of $G$ is the least number of edges that must be removed to $G$ to obtain a subgraph in $G$ that is not connected. Two graphs $G$ and $H$ are \emph{isomorphic} when there exists a bijective mapping $\varphi$ between $V(G)$ and $V(H)$ such that $uv$ is in $E(G)$ if and only if $\varphi(u)\varphi(v)$ is in $E(H)$. Such a function $\varphi$ is called an \emph{isomorphism}. 
For each positive integer $n$, we denote $p_n(G)$ the number of subgraphs in $G$ that are isomorphic to $P_n$. 

A \emph{two-terminal graph} is a graph $G$ equipped with two distinguished vertices called \emph{terminals}. In what follows, we will use the symbols $s$ and $t$ to denote the terminal vertices of any two-terminal graph. A pair of two-terminal graphs $G$ and $H$ is \emph{isomorphic} if there exists an isomorphism $\varphi$ between the graphs $G$ and $H$ that preserves the set of terminal vertices. For each positive integer $d$ and each two-terminal graph $G$, a spanning subgraph $H$ of $G$ is a \emph{$d$-pathset} if $H$ has some path of length at most $d$ joining its terminals. We say a two-terminal graph $G$ in $T_{n,m}$ is 
\emph{complete} when the connected component that includes both terminals is a complete graph. 

\section{Related work}\label{section:background}
Let $G$ be any two-terminal graph in $T_{n,m}$ and let $d$ be any positive integer. For each $i$ in $\{1,2,\ldots,m\}$, let $N_i^d(G)$ be the number of $d$-pathsets of $G$ having precisely $i$ edges. For every $\rho$ in $[0,1]$, the $d$-constrained two-terminal reliability of $G$ can be written as follows:
\begin{equation}\label{equation:poly}
R_G^d(\rho) = \sum_{i=1}^{m}N_i^d(G)(1-\rho)^i\rho^{m-i}.    
\end{equation}

Sometimes we will consider for convenience the \emph{$d$-constrained two-terminal unreliability of $G$}, denoted $U_G^d(\rho)$, which is defined as $1-R_G^d(\rho)$. The key concept of this work is the following. 

\begin{definition}
Let $d$ be any positive integer. The two-terminal graph $G$ in $T_{n,m}$ is a \emph{$d$-uniformly most reliable two-terminal graph} ($d$-UMRTTG) if for each $H$ in $T_{n,m}$ and every $\rho$ in $[0,1]$ it holds that $R_G^d(\rho)\geq R_H^d(\rho)$.     
\end{definition}

Given positive integers $d$, $n$ and $m$ such that $n\geq 2$ and $1\leq m\leq \binom{n}{2}$, we want to know if there exists a $d$-UMRTTG in $T_{n,m}$. The following concept gives a sufficient condition for a two-terminal graph $G$ in $T_{n,m}$ to be $d$-UMRTTG.

\begin{definition}
Let $G$ and $H$ be in $T_{n,m}$ and let $d$ be a positive integer. We say \emph{$G$ is $d$-stronger than $H$} if for each $i\in \{1,2,\ldots,n\}$ it holds that $N_i^d(G)\geq N_i^d(H)$ and further, 
$N_j^d(G)>N_j^d(H)$ for some $j\in \{1,2,\ldots,n\}$.    
\end{definition}

Let $G$ be in $T_{n,m}$. By equation~\eqref{equation:poly}, if $G$ is $d$-stronger than any other graph $H$ in $T_{n,m}$ then $G$ is a $d$-UMRTTG. The following concepts give necessary conditions for a two-terminal graph $G$ in $T_{n,m}$ to be a $d$-UMRTTG.

\begin{definition}
Let $G$ be a two-terminal graph in $T_{n,m}$ and let $d$ be any positive integer. 
\begin{enumerate}[label=(\roman*)]
\item We say $G$ is a \emph{$d$-locally most reliable two-terminal graph near $0$} ($d$-LMRTTG near $0$) if there exists $\delta>0$ such that for each $H$ in $T_{n,m}$ and every $\rho$ in $(0,\delta)$ it holds that $R_G^d(\rho)\geq R_H^d(\rho)$.     
\item We say $G$ is a \emph{$d$-locally most reliable two-terminal graph near $1$} ($d$-LMRTTG near $1$) if there exists $\delta>0$ such that for each $H$ in $T_{n,m}$ and every $\rho$ in $(1-\delta,1)$ it holds that $R_G^d(\rho)\geq R_H^d(\rho)$. 
\end{enumerate}
\end{definition}

Clearly, each $d$-UMRTTG is simultaneously $d$-LMRTTG near $0$ and $d$-LMRTTG near $1$. Remark~\ref{remark:local} is an essential tool to find $d$-LMRTTGs near $0$ or $1$. Its proof follows directly from elementary calculus and equation~\eqref{equation:poly}.
\begin{remark}\label{remark:local}
Let $d$ be any positive integer and let $G$ and $H$ be two-terminal graphs in $T_{n,m}$.
\begin{enumerate}[label=(\roman*)]
    \item\label{it1:local} If $N_i^d(G)=N_i^d(H)$ for all $i\in \{1,2,\ldots,k-1\}$ and  
    $N_k^d(G)>N_k^d(H)$, then there exists $\delta>0$ such that for every $\rho\in (1-\delta,1)$ it holds that $R_G^d(\rho)> R_H^d(\rho)$.
    \item\label{it2:local} If $N_{i}^d(G)=N_{i}^d(H)$ for all $i\in \{j+1,j+2,\ldots,m\}$ and  
    $N_{j}^d(G)>N_{j}^d(H)$, then there exists $\delta>0$ such that for every $\rho\in (0,\delta)$ it holds that $R_G^d(\rho)> R_H^d(\rho)$.
\end{enumerate}
\end{remark}

Observe that, for each nonempty set $T_{n,m}$ and each positive integer $d$, there always exists $G$ in $T_{n,m}$ that is a $d$-LMRTTG near $0$, or near $1$. Let us illustrate how to construct the set consisting of all $d$-LMRTTGs near $1$ (the reasoning to construct all $d$-LMRTTGs near $0$ is analogous). 
Define $T_{n,m}^{d}(1)$ as $T_{n,m}$ and, for each $i\in \{1,2,\ldots,m-1\}$ define $T_{n,m}^d(i+1)$ as follows,
\begin{equation*}
T_{n,m}^d(i+1)=\{G: G\in T_{n,m}^d(i), \, N_{i+1}^d(G)\geq N_{i+1}^d(H) \text{ for each } H \text{ in } T_{n,m}^d(i)\}.
\end{equation*}

\begin{lemma}\label{lemma:existenceLMRTTG}
For each positive integer $d$ and each nonempty set $T_{n,m}$, the set consisting of all $d$-LMRTTGs 
near $1$ is nonempty and equals $T_{n,m}^{d}(m)$.
\end{lemma}
\begin{proof}
As $T_{n,m}$ is nonempty and we are maximizing a function over a finite set,  $T_{n,m}^d(1)$ as well as each of the sets $T_{n,m}^d(2),T_{n,m}^d(3),\ldots,T_{n,m}^{d}(m)$ is nonempty. 
Now, we will prove that $T_{n,m}^d(m)$ consists precisely of all $d$-LMRTTGs near $1$ in $T_{n,m}$. 
Let $G$ be in $T_{n,m}^d(m)$ and $H$ in $T_{n,m}-T_{n,m}^d(m)$. As $H$ is not in $T_{n,m}^d(m)$, 
there exists $k$ in $\{1,2,\ldots,m\}$ such that $N_i^d(G)=N_i^d(H)$ for each $i\in \{0,1,\ldots,k-1\}$ but $N_k^d(G)>N_k^d(H)$. The statement then follows from Remark~\ref{remark:local}\ref{it1:local}.
\end{proof}

\begin{remark}\label{remark:uniqueness}
If $T_{n,m}^d(m)$ consists of a single two-terminal graph $G$  then there exists $\delta>0$ such that 
for each $H$ in $T_{n,m}-\{G\}$ and every $\rho$ in $(1-\delta,1)$ 
it holds that $R_G^d(\rho)>R_H^d(\rho)$. In such a case, either $G$ is the only $d$-UMRTTG in $T_{n,m}$ or there exists $\rho_0$ in $(0,1)$ and $H$ in $T_{n,m}$ such that $R_{H}^d(\rho_0)>R_G^d(\rho_0)$, in which case no $d$-UMRTTG exists in $T_{n,m}$.     
\end{remark}

The existence or nonexistence of $d$-UMRTTGs has only been studied when the distance constrained is dropped, that is, when $d\geq n-1$. In the following, we denote LMRTTG, UMRTTG, $N_i(G)$, or $R_{G}(\rho)$ for the respective symbols $d$-LMRTTG, $d$-UMRTTG, $N_i^d(G)$ or $R_{G}^d(\rho)$ whenever 
$d\geq n-1$ (i.e., the prefix $d$ in the words or acronyms, or the superscript $d$ in the mathematical symbols, will not appear when $d\geq n-1$). 

The study of existence or nonexistence of UMRTTGs was pioneered by Bertrand et al.~\cite{2018Bertrand} and by Sun Xie and et al.~\cite{2021Xie}. In the remaining of this section we will give a list of results appearing in those works that will be useful for our purposes. First, a couple of definitions will be essential to understand such results. 

\begin{definition}
Let $n$ and $r$ be integers such that $n\geq 3$ and $r\in \{0,1,\ldots,n-3\}$. The two-terminal graph $A_{n,r}$ has vertex set $\{s,t,v_3,\ldots,v_n\}$ and edge set $\{st\} \cup \{sv_i,v_it, i\in \{3,4,\ldots,n\}\} \cup \{v_3v_{3+j}, j\in \{1,2,\ldots,r\}\}$.
\end{definition}

\begin{definition}
Let $n$ and $m$ be integers such that $n\geq 5$ and $5\leq m \leq 2n-3$. Define $H_{n,m}$ as the two-terminal graph in $T_{n,m}$ that arises from $A_{\frac{m+2}{2},1}$ by the addition of $n-\frac{m+2}{2}$ isolated vertices when $m$ is even, or 
the two-terminal graph in $T_{n,m}$ that arises from $A_{\frac{m+3}{2},0}$ by the addition of $n-\frac{m+3}{2}$ isolated vertices when $m$ is odd.
\end{definition}

The two-terminal graph $A_{n,r}$ is depicted in Figure~\ref{anr}. 
\begin{figure}
    \centering
\begin{tikzpicture}
\filldraw[black] (4.5,2) circle (2pt); 
\node[above] (s) at (4.5,2) {$v_2=t$}; 

\filldraw[black] (0,0) circle (2pt); 
\node[below] at (0,0) {$v_3$}; 
\filldraw[black] (1,0) circle (2pt); 
\node[below] at (1,0) {$v_4$}; 
\filldraw[black] (2,0) circle (2pt); 
\node[below] at (2,0) {$v_5$}; 

\filldraw[black] (2.6,0) circle (.5pt); 
\filldraw[black] (3,0) circle (.5pt); 
\filldraw[black] (3.4,0) circle (.5pt); 

\filldraw[black] (3.6,1) circle (.5pt); 
\filldraw[black] (3.8,1) circle (.5pt); 
\filldraw[black] (4,1) circle (.5pt); 
\filldraw[black] (3.6,-1) circle (.5pt); 
\filldraw[black] (3.8,-1) circle (.5pt); 
\filldraw[black] (4,-1) circle (.5pt); 

\filldraw[black] (4,0) circle (2pt); 
\node[below] at (4,0) {$v_{r+2}$}; 
\filldraw[black] (5,0) circle (2pt); 
\node[below] at (5,0) {$v_{r+3}$}; 
\filldraw[black] (6,0) circle (2pt); 
\node[below] at (6,0) {$v_{r+4}$}; 
\filldraw[black] (7,0) circle (2pt); 
\node[below] at (7,0) {$v_{r+5}$}; 

\filldraw[black] (7.6,0) circle (.5pt); 
\filldraw[black] (8,0) circle (.5pt); 
\filldraw[black] (8.4,0) circle (.5pt); 

\filldraw[black] (6,1) circle (.5pt); 
\filldraw[black] (6.2,1) circle (.5pt); 
\filldraw[black] (6.4,1) circle (.5pt); 
\filldraw[black] (6,-1) circle (.5pt); 
\filldraw[black] (6.2,-1) circle (.5pt); 
\filldraw[black] (6.4,-1) circle (.5pt); 

\filldraw[black] (9,0) circle (2pt); 
\node[below] at (9,0) {$v_{n}$}; 

\filldraw[black] (4.5,-2) circle (2pt); 
\node[below] (s) at (4.5,-2) {$v_1=s$}; 

\draw (4.5,2) -- (0,0);
\draw (4.5,2) -- (1,0);
\draw (4.5,2) -- (2,0);
\draw (4.5,2) -- (4,0);
\draw (4.5,2) -- (5,0);
\draw (4.5,2) -- (6,0);
\draw (4.5,2) -- (7,0);
\draw (4.5,2) -- (9,0);

\draw (4.5,-2) -- (0,0);
\draw (4.5,-2) -- (1,0);
\draw (4.5,-2) -- (2,0);
\draw (4.5,-2) -- (4,0);
\draw (4.5,-2) -- (5,0);
\draw (4.5,-2) -- (6,0);
\draw (4.5,-2) -- (7,0);
\draw (4.5,-2) -- (9,0);

\draw (0,0) -- (1,0);
\draw (0,0) to[bend left] (2,0);
\draw (0,0) to[bend left] (4,0);
\draw (0,0) to[bend left] (5,0);

\filldraw[black] (1.2,.35) circle (.5pt); 
\filldraw[black] (1.3,.4) circle (.5pt); 
\filldraw[black] (1.4,.45) circle (.5pt); 

\draw (4.5,2) to[out=0,in=90] (9.5,0);
\draw (9.5,0) to[out=-90,in=0] (4.5,-2);
\end{tikzpicture}
    \caption{\centering Two-terminal graph $A_{n,r}$}
    \label{anr}
\end{figure}
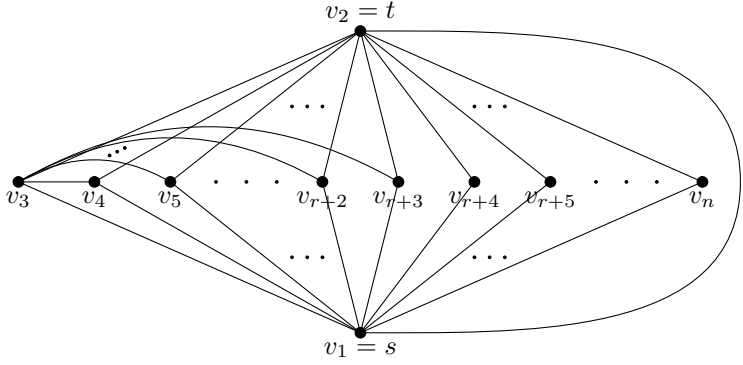
Bertrand et al.~\cite{2018Bertrand} used the construction stated in Lemma~\ref{lemma:existenceLMRTTG} to determine the only LMRTTG near $1$ in $T_{n,m}$ for some pairs of positive integers $n$ and $m$.
\begin{lemma}[Bertrand et al.~\cite{2018Bertrand}]\label{lemma:LMRTTG1}
Let $n$ and $m$ be integers such that $n\geq 5$ and $m\geq 5$. The following assertions hold.
\begin{enumerate}[label=(\roman*)]
\item If $m\leq 2n-3$ then $H_{n,m}$ is the only LMRTTG near $1$ in $T_{n,m}$.
\item If $2n-3<m\leq 3n-6$ then $A_{n,m-2n+3}$ is the only LMRTTG near $1$ in $T_{n,m}$. 
\end{enumerate}
\end{lemma}

Bertrand et al.~\cite{2018Bertrand} succeeded to find, for several pairs of integers $n$ and $m$ considered in Lemma~\ref{lemma:LMRTTG1}, a two-terminal graph in $T_{n,m}$ whose two-terminal reliability is greater than that of the only LMRTTG near $1$ when $\rho=1/2$. 

\begin{lemma}[Bertrand et al.~\cite{2018Bertrand}]\label{lemma:half}
Let $n$ and $m$ be integers such that $n\geq 11$ and $20\leq m \leq 3n-9$. Define $n'=\lceil \frac{n}{3}\rceil +2$, $r'=m-2n'+3$, and the two-terminal graph $G_{n,m}$ in $T_{n,m}$ given by $A_{n',r'}\cup \overline{K_{n-n'}}$. The following assertions hold:
\begin{enumerate}[label=(\roman*)]
\item\label{half1} If $m\leq 2n-3$ then $R_{G_{n,m}}(1/2)>R_{H_{n,m}}(1/2)$.
\item\label{half2} If $2n-3<m\leq 3n-9$ then $R_{G_{n,m}}(1/2) > R_{A_{n,m-2n+3}}(1/2)$.
\end{enumerate}
\end{lemma}

The nonexistence of UMRTTGs in $T_{n,m}$ for each of the pairs $n$ and $m$ stated in Lemma~\ref{lemma:half} follows from Remark~\ref{remark:uniqueness}. The authors summarize their findings in Theorem~\ref{theorem:bertrand}.
\begin{theorem}[Bertrand et al.~\cite{2018Bertrand}]\label{theorem:bertrand}
For each pair of integers $n$ and $m$ such that $n\geq 11$ and $20 \leq m \leq 3n-9$ there is no UMRTTG in $T_{n,m}$.
\end{theorem}

An elegant result on the nonexistence of UMRTTGs was given by Xie et al.~\cite{2021Xie}. In the following, for each $G$ in $T_{n,m}$ we let $\hat{G}$ be the simple graph $G-s-t$. 
\begin{theorem}[Xie et al.~\cite{2021Xie}]\label{theorem:Xie}
Let $n$ and $m$ be integers such that $n\geq 11$ and $3n-6<m\leq \binom{n}{2}-2$. Let $G$ be any two-terminal graph in $T_{n,m}$. The following assertions hold.
\begin{enumerate}[label=(\roman*)]
\item\label{teo-regular} If $G$ is a LMRTTG near $0$ then the terminals in $G$ are universal and $\hat{G}$ is almost-regular. 
\item\label{teo-not-regular} If $G$ is a LMRTTG near $1$ then the terminals in $G$ are universal and $\hat{G}$ is not almost-regular.
\end{enumerate}
In particular, there is no UMRTTG in $T_{n,m}$ when $n\geq 11$ and $3n-6<m\leq \binom{n}{2}-2$.
\end{theorem}

To close this section, we will briefly explain the methods employed by Xie et al. to prove Theorem~\ref{theorem:Xie}.
Let $n$ and $m$ be integers such that $n\geq 11$ and $3n-6< m \leq \binom{n}{2}-2$. Denote $T_{n,m}^{u}$ the collection of two-terminal graphs in $T_{n,m}$ such that the terminals $s$ and $t$ are universal. 

On the one hand, the authors proved that if $G$ is a LMRTTG near $0$ then $G$ is in $T_{n,m}^u$ and the graph $\hat{G}$ has the greatest edge-connectivity. It is known~\cite{1962Harary} that if a graph has the greatest edge-connectivity then it is almost-regular, which gives Theorem~\ref{theorem:Xie}\ref{teo-regular}.  

On the other hand, the authors used the construction given in Lemma~\ref{lemma:existenceLMRTTG} to prove that $T_{n,m}(3)$ equals the set $T_{n,m}^u$. 
Then, the authors found the following expression for $N_4(G)$ for each $G$ in $T_{n,m}^{u}$:
\begin{equation*}
N_4(G)= \binom{m-1}{3} + (n-2)\binom{m-3}{2}-\binom{n-2}{2}+2(m-2n+3)(m-6)+2p_3(\hat{G}). 
\end{equation*}
Byer~\cite{1999Byer} proved that each graph $\hat{G}$ maximizing $p_3(\hat{G})$ belongs to at least one of the six classes $\mathcal{C}^1_{n,m},\ldots,\mathcal{C}^6_{n,m}$ (the interested reader can find the definition of these six classes in~\cite{1999Byer}). Later, the authors of~\cite{2021Xie} checked, for each $i\in \{1,2,\ldots,6\}$, that each graph in $\mathcal{C}^i_{n,m}$ is not almost-regular. Consequently, the set $T_{n,m}(4)$ consists of some two-terminal graphs $G$ in $T_{n,m}^u$ such that $\hat{G}$ is not almost-regular. As $T_{n,m}(m)$ is included in the set $T_{n,m}(4)$, 
Theorem~\ref{theorem:Xie}\ref{teo-not-regular} then follows by Lemma~\ref{lemma:existenceLMRTTG}. A recent breakthrough in the study of LMRTTGs near $1$ was given by Gong and Lin~\cite{2024Gong}. They achieved in characterizing all LMRTTGs near $1$ in a large number of nonempty classes $T_{n,m}$ such that $n\geq 6$ and $2n-3\leq m\leq \binom{n}{2}$. 

We remark that Bertrand et al.~\cite{2018Bertrand} as well as Xie et al.~\cite{2021Xie} also proved the nonexistence of UMRTTGs in some nonempty sets $T_{n,m}$ such that $n\leq 10$, but these cases can be obtained by computation means.

\section{Existence of 1-UMRTTGs and 2-UMRTTGs}\label{section:1and2}
Here we will characterize all $1$-UMRTTGs and $2$-UMRTTGs. A key concept is that of irrelevant edges. 
\begin{definition}\label{definition:irrelevance}
Let $G \in T_{n,m}$ and let $d$ be a positive integer. 
An edge $e$ in $G$ is \emph{$d$-irrelevant} if for each $d$-pathset $H$ in $G$ 
including the edge $e$ it holds that $H-e$ is also a $d$-pathset in $G$. 
\end{definition}

Lemma~\ref{lemma:irrelevant} follows from Definition~\ref{definition:irrelevance} and 
Equation~\eqref{equation:poly}. 
\begin{lemma}\label{lemma:irrelevant}
Let $G\in T_{n,m}$ and let $d$ be a positive integer. 
If an edge $e$ in $G$ is $d$-irrelevant then $R_{G-e}^d(\rho)=R_G^d(\rho)$ for every $\rho\in [0,1]$. 
\end{lemma}
\begin{proof}
Let $G$ be any two-terminal graph in $T_{n,m}$. As $e$ is $d$-irrelevant in $G$, $e$ is not included in any path of length at most $d$ joining the terminals in $G$. Fix $i\in \{1,2,\ldots,m-1\}$. Let $H$ be any $d$-pathset in $G$ composed by $i$ edges.  
If $e$ is in $H$ then, by our assumption, $H-e$ is still a $d$-pathset in $G-e$ with $i-1$ edges. If $e$ is not in $H$ then $H$ is already a $d$-pathset in $G-e$. By the sum-rule we get that 
$N_i^d(G)=N_{i-1}^d(G-e)+N_i(G-e)$. Additionally, $N_{m}^d(G)=N_{m-1}^d(G-e)$. Replacing into equation~\eqref{equation:poly} and letting $N_0^d(G-e)=0$ gives that, for every $\rho$ in $[0,1]$,
\begin{align*}
R_{G}^d(\rho) &= \sum_{i=1}^{m}N_i^d(G)(1-\rho)^i\rho^{m-i}   
    = \sum_{i=1}^{m}N_{i-1}^d(G-e)(1-\rho)^i\rho^{m-i} + 
      \sum_{i=1}^{m-1}N_i^d(G-e)(1-\rho)^i\rho^{m-i}\\
    &= (1-\rho)\sum_{i=1}^{m}N_{i-1}^d(G-e)(1-\rho)^{i-1}\rho^{(m-1)-(i-1)} + \rho \sum_{i=1}^{m-1}N_i^d(G-e)(1-\rho)^{i}\rho^{m-1-i}\\
     &= (1-\rho)R^d_{G-e}(\rho) + \rho R_{G-e}^d(\rho)=R^d_{G-e}(\rho). \qedhere
\end{align*} 
\end{proof}

A characterization of all $1$-UMRTTGs is a consequence of Lemma~\ref{lemma:irrelevant}. 

\begin{corollary}\label{cor:d1}
Let $n$ and $m$ be integers such that $n\geq 2$ and $1\leq m \leq \binom{n}{2}$. 
A two-terminal graph $G$ in $T_{n,m}$ is a $1$-UMRTTG if and only if $st$ is in $G$.    
\end{corollary}
\begin{proof}
Let $G$ be any two-terminal graph in $T_{n,m}$. 
Any edge $e$ in $G$ that is not $st$ is $1$-irrelevant thus we can delete $e$ and the $1$-constrained two-terminal reliability is preserved. On the one hand, if $st$ is not in $G$ then we can remove all edges in $G$ and $R_G^d(\rho)=0$ for every $\rho$ in $[0,1]$. On the other hand, if $st$ is in $G$ 
then $R_G^d(\rho)$ equals the probability that $st$ does not fail, which is $1-\rho$.
\end{proof}

Lemma~\ref{lemma:st} roughly states that $st$ is the most relevant edge regardless of the distance constraint.
\begin{lemma}[Presence of the edge $st$]\label{lemma:st}
For each $G$ in $T_{n,m}$ such that $st\notin G$, each edge $e$ in $G$, and each positive integer $d$, the two-terminal graph $G-e+st$ is $d$-stronger than $G$.
\end{lemma}
\begin{proof}
Let $G$ be in $T_{n,m}$ such that $st\notin G$, and let $e$ be an edge in $G$. Define $H$ as $G-e+st$. On the one hand, each $d$-pathset in $G$ not including $e$ is a $d$-pathset in $H$. On the other hand, for each $d$-pathset $G'$ in $G$ including $e$ we can assign the $d$-pathset $H'$ in $H$ given by $G'-e+st$. Consequently, for each $i\in \{2,3,\ldots,m\}$ 
it holds that $N_i^d(H)\geq N_i^d(G)$. As $N_1^d(H)=1$ and $N_1^d(G)=0$, the lemma follows.
\end{proof}

We are in position to characterize all $2$-UMRTTGs.
\begin{proposition}\label{proposition:2UMRTTG}
Let $n$ and $m$ be integers such that $n\geq 2$ and $1\leq m \leq \binom{n}{2}$. A two-terminal graph $G$ in $T_{n,m}$ is a $2$-UMRTTG if and only if $G$ includes $A_{n',0}$ as a subgraph, 
where $n' = \min\{\lfloor \frac{m+3}{2}\rfloor,n\}$.
\end{proposition}
\begin{proof}
Let $n$ and $m$ be integers as in the statement. Let $G$ by any two-terminal graph in $T_{n,m}$. As the $2$-constrained two-terminal reliability is nondecreasing with the addition of edges, $R_{G}(\rho) \leq R_{K_n}^2(\rho)$ for every $\rho$ in $[0,1]$, where $K_n$ is the complete graph equipped with two terminals. Each edge in $K_n$ that is not a $1$-path neither a $2$-path is $2$-irrelevant. Then, by Lemma~\ref{lemma:irrelevant}, the $2$-constrained two-terminal reliability of $K_n$ equals that of $A_{n,0}$. If $m\geq 2n-3$ then, for every $\rho$ in $[0,1]$, it holds that 
$R_G(\rho)\leq R_{K_n}^2(\rho)=R_{A_{n,0}}^2(\rho)$, and the equality holds if and only if $G$ includes $A_{n,0}$ as a subgraph. 

Now, assume that $m<2n-3$. Let $n'=\lfloor \frac{m+3}{2} \rfloor$. 
Let $G$ be any two-terminal graph in $T_{n,m}$ that includes $A_{n'0}$ as a subgraph and let $H$ be any two-terminal graph in $T_{n,m}$. 
It is enough to prove that $U_{G}^2(\rho)\leq U_H^2(\rho)$ for every $\rho$ in $[0,1]$.  
By Lemma~\ref{lemma:st}, we can assume without loss of generality that $st$ is in $H$. As each edge in $G$ that is not a $1$-path neither a $2$-path is $2$-irrelevant, 
after the deletion of $2$-irrelevant edges in $G$ we get that 
$U_G^2(\rho)=U_{A_{n',0}}^2(\rho)$ and that $U_H^2(\rho)=U_{A_{\ell,0}}^2(\rho)$ for every $\rho$ in $[0,1]$ and for some nonnegative integer $\ell$. As $A_{\ell,0}$ has at most as many edges as $H$, 
we know that $2\ell-3\leq m$ thus $\ell\leq n'$. The failure probability of each $2$-path joining the terminals $s$ and $t$ equals $(1-(1-\rho)^2)$. Consequently, for every $\rho$ in $(0,1)$,
\begin{equation*}
U_G^2(\rho)=U_{A_{n',0}}^2(\rho) = \rho(1-(1-\rho)^2)^{n'}\leq \rho(1-(1-\rho)^2)^{\ell}= U_{A_{\ell,0}}^2(\rho) = U_H^2(\rho).\qedhere
\end{equation*}
\end{proof}

\section{Existence of 3-UMRTTGs}\label{section:3}
In this section we will prove that, for each pair of integers $n$ and $m$ such that $n\geq 6$ and $5\leq m \leq 2n-3$, the only 3-UMRTTG in $T_{n,m}$ 
is $H_{n,m}$. 

Let $T_{n,m}^*$ the set consisting of all two-terminal graphs in $T_{n,m}$ 
without $3$-irrelevant edges whose terminals are true twins. In a first step we prove Lemma~\ref{lemma:conditions}, which states that for each $G$ in $T_{n,m}-T_{n,m}^*$ there exists $H$ in $T_{n,m}^*$ that is 
$3$-stronger than $G$. In a second step we prove that for each $G$ in $T_{n,m}^*-\{H_{n,m}\}$ and every $\rho\in (0,1)$ 
it holds that $R_{H_{n,m}}^3(\rho)>R_G^3(\rho)$. 

\begin{lemma}\label{lemma:notcomplete}
Let $n$, $m$ and $d$ be integers such that $n\geq 5$, $m\geq 5$, and $d\geq 3$.  
If $G$ is in $T_{n,m}$, $v$ is a vertex hanging to $s$, and $G-sv$ is not complete,  
then there exists $H$ in $T_{n,m}$ that is $d$-stronger than $G$.
\end{lemma}
\begin{proof}
In the conditions of the statement there exists no path including $sv$ joining the terminals 
thus $sv$ is $d$-irrelevant. We will construct a two-terminal graph $H$ in $T_{n,m}$ that is $d$-stronger than $G$.
If there exists $v_i$ in $G$ such that $sv_i$ is in $G$ but $v_it$ is not in $G$ then we define $H$ as $G-sv+sv_i$. Similarly, if $sv_i\notin G$ and $v_it \in G$ then we define $H$ as 
$G-sv+sv_i$. Else, if there exists a pair of nonadjacent vertices $v_i$ and $v_j$ that are common neighbors of the terminals then we define $H$ as $G-sv+v_iv_j$. Otherwise, the connected component induced by the terminals and its common neighbors is a complete graph. As 
$G-sv$ is not a complete graph, there exists an additional edge $e$ in $G$ that is not included in the connected component that includes the terminals. Consequently, $e$ is $d$-irrelevant and 
we can define $H$ as $G-st+e$. In either case, $H$ arises from $G$ by the deletion of a $d$-irrelevant edge and the addition of an edge that creates a new $2$-path, or a new $3$-path, joining the terminals in $H$. Consequently, $H$ is $d$-stronger than $G$, as required.
\end{proof}

\begin{lemma}\label{lemma:complete}
Let $n$, $m$ and $d$ be integers such that $n\geq 5$ and $m\geq 5$.  
If $G$ is in $T_{n,m}$, $v$ is a vertex hanging to $s$ and $G-sv$ is complete, then there exists $H$ in $T_{n,m}$ that is $3$-stronger than $G$.
\end{lemma}
\begin{proof}
As $G-sv$ is a complete graph on at least $4$ edges, it has at  least $4$ vertices. Pick two vertices $v_i$ and $v_j$ that are common neighbors to both terminals in $G$, and define $H$ as $G-v_iv_j+vt$. As $sv,vt\in H$ and $d\geq 3$, it is clear that $N_2^3(H)=N_2^3(G)+1$. Let $i\in \{1,2,\ldots,m\}$. It is enough to prove that $N_i^3(H)\geq N_i^3(G)$. Let $\mathcal{P}_i^3(G)$ and $\mathcal{P}_i^3(H)$ be the classes of all $3$-pathsets of $G$ and $H$ with $i$ edges, respectively. Define the function $f_i:\mathcal{P}_i^3(G) \to \mathcal{P}_i^3(H)$ as follows:
\begin{equation*}
f_i(G') = 
\begin{cases}
        G', & \text{if } v_iv_j\notin G'\\
        G'-v_iv_j+vt, & \text{if } v_iv_j\in G' \text{ and } G'-v_iv_j\in \mathcal{P}_{i-1}^3(G)\\
        G'-v_iv_j+vt, & \text{if } v_iv_j\in G', G'-v_iv_j\notin \mathcal{P}_{i-1}^3(G) \text{ and } sv\in G'\\
        G'-v_iv_j+vt-sv_i+sv_j, & \text{if } v_iv_j\in G', G'-v_iv_j\notin \mathcal{P}_{i-1}^3(G), sv\notin G' \text{ and } sv_i,v_jt\in G'\\
        G'-v_iv_j+vt-sv_j+sv, & \text{if } v_iv_j\in G', G'-v_iv_j\notin \mathcal{P}_{i-1}^3(G), sv\notin G' \text{ and } sv_j,v_it\in G'    
\end{cases}        
\end{equation*}
Observe that if $G'$ is in $\mathcal{P}_i^3(G)$ but 
$G'-v_iv_j$ is not in $\mathcal{P}_{i-1}^3(G)$, then 
either $sv_i,v_jt\in G'$ or $sv_j,v_it\in G'$. Therefore, the five cases in the definition of $f_i$ covers the domain of $3$-pathsets in $G$ with $i$ edges and $f_i$ is well-defined. By construction, $f_i$ is injective. Then, $N_i^3(H)=|\mathcal{P}_i^3(H)|\geq |\mathcal{P}_i^3(G)| = N_i^3(G)$ and $H$ is $3$-stronger than $G$, as required. 
\end{proof}

\begin{lemma}\label{lemma:terminalneighbors}
Let $n$, $m$ and $d$ be integers such that $n\geq 5$ and $m\geq 5$. For each $G$ in $T_{n,m}$ whose terminals are not true twins there exists $H$ in $T_{n,m}$ that is $3$-stronger than $G$ whose terminals are true twins.
\end{lemma}
\begin{proof}
Let $G$ be as in the statement. By Lemma~\ref{lemma:st} we assume without loss of generality that $st$ is in $G$. As the terminals in $G$ are not true twins, let us assume that there exists $v$ in $G$ such that 
$sv \in G$ but $vt\notin G$ (the other case is analogous). If $deg_G(v)=1$  then $v$ is a vertex hanging to $s$. Either if $G-sv$ is complete or not, by Lemma~\ref{lemma:notcomplete} and Lemma~\ref{lemma:complete} there exists $H$ in $T_{n,m}$ that is $3$-stronger than $G$. Therefore, we can assume from now on that $deg_G(v)\geq 2$, so $sv, vw \in G$ for some vertex $w$ in $G$. 

Define $H$ as $G-vw+vt$. As $H$ is in $T_{n,m}$, it is enough to prove that $H$ is $3$-stronger than $G$. Let $i\in \{1,2,\ldots,m\}$. We will prove that $N_i^3(H)\geq N_i^3(G)$. 
As $N_2^3(G')=N_2^3(G)+1$, the lemma will follow. 
For each $3$-pathset $G'$ of $G$ we define $H'$ as $G'$ 
if $vw\notin G'$ or $G'-vw+vt$ otherwise. 
Clearly, if $vw\notin G'$ then $H'$ is a $3$-pathset in $H$. 
Now, if $vw\in G'$ then let us consider two different cases. 
If $G'-vw \in \mathcal{P}_{i-1}^3(G)$ then clearly $H'=G'-vw+vt\in \mathcal{P}_i^3(G')$. Otherwise, if $G'-vw\notin \mathcal{P}_{i-1}^3(G)$ then $vw$ belongs to some $3$-path in $G'$. As $vt\notin G'$, the only $3$-path in $G'$ including $vw$ is $sv,vw,wt$. Consequently, $sv,vw,wt\in G'$ and $H'=G'-vw+vt$ includes the $2$-path $sv,vt$. We conclude that $H' \in \mathcal{P}_{i}^{3}(G')$. As the function $f:\mathcal{P}_{i}^{3}(G) \to \mathcal{P}_i^3(H)$ such that 
$f_i(G')=H'$ is injective, we conclude that  $N_i^3(G)=|\mathcal{P}_i^3(H)|\geq |\mathcal{P}_i^3(G)|= N_i^3(G)$, and the lemma follows.
\end{proof}

\begin{lemma}\label{lemma:allrelevant}
Let $n$, $m$ and $d$ be integers such that $n\geq 5$, $m\geq 5$, and $d\geq 3$. 
For each $G$ in $T_{n,m}$ such that $s$ and $t$ are true twins with some irrelevant edge there exists $H$ in $T_{n,m}^*$ that is $3$-stronger than $G$.
\end{lemma}
\begin{proof}
Let $G$ be in $T_{n,m}$ satisfying all the conditions of the statement, and let $uv$ be some irrelevant edge in $G$. 
As the terminals in $G$ are true twins, all edges incident to $s$ or to $t$ belong to some $1$-path or some $2$-path joining $s$ and $t$ (the edge $st$ is incident to both vertices and is a $1$-path). 
Consequently, the endpoints $u$ and $v$ belong to $G-s-t$. Observe that $u$ and $v$ cannot be common neighbors of $s$ and $t$, as in that case we would have that $su,uv,vt$ would be a $3$-path including 
the edge $e$, which contradicts that $e$ is irrelevant. We can assume without loss of generality that $u$ is neither a neighbor of $s$ nor $t$. If $deg_G(u)=1$ and $G-uv$ is complete then the statement follows from Lemma~\ref{lemma:complete}. If $deg_G(u)=1$ and $G-uv$ is not complete 
then the statement follows from Lemma~\ref{lemma:notcomplete}. Finally, if $deg_G(u)\geq 2$, let us pick two edges $e$ and $e'$ incident to $u$. As $u$ is neither a neighbor of $s$ nor $t$, both edges $e$ and $e'$ are $3$-irrelevant. The graph $G_1$ given by $G-e-e'+su+ut$ has all the $3$-pathsets of $G$ but also the $2$-path $su,ut$, hence $G_1$ is $3$-stronger than $G$. Observe that $G_1$ is in $T_{n,m}$ and its terminals are true twins and $G_1$ has fewer $3$-irrelevant edges than $G$. 

Let $G_0$ be the graph $G$. After repeated application of the previous process we will obtain a sequence of graphs $G_0,G_1,\ldots, G_r$ in $T_{n,m}$ whose terminals are true twins such that $G_r$ has no $3$-irrelevant edges. The lemma follows by letting $H$ as $G_r$. 
\end{proof}

\begin{lemma}\label{lemma:conditions}
Let $n$ and $m$ be integers such that $n\geq 5$ and $m\geq 5$. 
For each $G$ in $T_{n,m}-T_{n,m}^*$ there exists $H$ in $T_{n,m}^*$ that is $3$-stronger than $G$.
\end{lemma}
\begin{proof}
Let $G$ be an arbitrary two-terminal graph in $T_{n,m}-T_{n,m}^*$. If $st$ is not in $G$ then we pick an edge $e$ in $G$ and define $G_1$ as $G-e+st$. By Lemma~\ref{lemma:st} we know that $G_1$ is $3$-stronger than $G$. Otherwise, we define $G_1$ as $G$. In either case, $st$ is in $G_1$. 
If $s$ and $t$ are true twins in $G_1$ then we set $G_2$ as $G_1$. Otherwise, 
by Lemma~\ref{lemma:terminalneighbors}, there exists $G_2$ in $T_{n,m}$ whose terminals are true twins that is $3$-stronger than $G_1$. 
In either case, the terminals in $G_2$ are true twins. Finally, if $G_2$ has no irrelevant edges then we set $H$ as $G_2$. Otherwise, by Lemma~\ref{lemma:allrelevant} there exists $H$ in $T_{n,m}^*$ 
that is $3$-stronger than $G_2$. 
Consequently, there exists $H$ in $T_{n,m}^{*}$ that is $3$-stronger than $G$, as required.
\end{proof}

The $3$-constrained two-terminal unreliability of $H_{n,m}$ only depends on $\rho$ and $m$.  
For short, we denote $U_m(\rho)$ for $U_{H_{n,m}}^3(\rho)$. 

\begin{lemma}\label{lemma:Um}
Let $n$ be any integer such that $n\geq 6$. The following assertions hold:
\begin{enumerate}[label=(\roman*)]    
    \item\label{it1} For each integer $m$ such that $5\leq m \leq 2n-5$ and every $\rho$ in $[0,1]$, $U_{m+2}(\rho)=(1-(1-\rho)^2)U_{m}(\rho)$.
    \item\label{it2}  For each integer $m$ such that $5\leq m \leq 2n-4$ and every $\rho$ in $[0,1]$, $U_{m+1}(\rho)\leq U_{m}(\rho)$.
    \item \label{it3} For each pair of integers $m$ and $x$ such that $9\leq m \leq 2n-3$ and  
    $1\leq x\leq \frac{m-7}{2}$ and every $\rho$ in $[0,1]$,
    $$U_m(\rho)\leq \rho^2U_{m-x-2}(\rho)+2(1-\rho)\rho^{x+1}\sum_{i=0}^x \binom{x}{i}(1-\rho)^{i}U_{m-2-x-i}(\rho).$$    
\end{enumerate}
\end{lemma}
\begin{proof} 
Let $n$ be any fixed integer such that $n\geq 6$. Let us proof each of the assertions separately.
\begin{enumerate}[label=(\roman*)]    
\item Observe that, for each integer $m$ such that $5\leq m \leq 2n-5$, 
$H_{n,m+2}$ is precisely $H_{n,m}$ plus a disjoint $2$-path joining $s$ and $t$. As the 3-constrained two-terminal unreliability of a $2$-path joining $s$ and $t$ equals $(1-(1-\rho)^2)$ and the edge failures are independent, we get that 
$U_{m+2}(\rho) = (1-(1-\rho)^2)U_{m}(\rho)$, as required.
\item We proceed by induction over $m\in \{5,6,\ldots,2n-4\}$. 
As $H_{n,5}\subseteq H_{n,6}$ we know that $U_6(\rho)\leq U_5(\rho)$. It is a simple exercise to prove that $U_7(\rho)\leq U_6(\rho)$ thus $U_7(\rho)\leq U_6(\rho)\leq U_5(\rho)$.  
Now, assume that $U_{h+2}(\rho)\leq U_{h+1}(\rho)\leq U_h(\rho)$ for some $h\geq 5$. It is enough to prove that $U_{h+3}(\rho)\leq U_{h+2}(\rho)$. 
As $h\geq 5$, by~\ref{it1} we know that $U_{h+3}(\rho)=(1-(1-\rho)^2)U_{h+1}(\rho)$ and by our inductive assumption, $U_{h+1}(\rho)\leq U_{h}(\rho)$. Consequently, $U_{h+3}(\rho)=(1-(1-\rho)^2)U_{h+1}(\rho)\leq (1-(1-\rho)^2)U_h(\rho)=U_{h+2}(\rho)$. 

\item As $m\geq 9$, by~\ref{it1} we know that $U_{m}(\rho)=(1-(1-\rho)^2)U_{m-2}(\rho) = \rho^2U_{m-2}(\rho)+2(1-\rho)\rho U_{m-2}(\rho)$.   
As $m-2\geq 5$ and $m-x-2\geq 5$, by~\ref{it2} we know that $U_{m-2}(\rho)\leq U_{m-x-2}(\rho)$ for every $\rho$ in $[0,1]$. Consequently, it is enough to prove that $U_{m-2}(\rho)\leq \rho^x\sum_{i=0}^{x}\binom{x}{i}(1-\rho)^iU_{m-2-x-i}(\rho)$. 
As $m-2-2x\geq 5$, $H_{n,m-2}$ is the union of $H_{n,m-2-2x}$ and $x$ edge-disjoint $2$-paths joining $s$ and $t$. Let $A$ 
be the vertex set in $H_{n,m-2}$ 
given by $\{v_3,v_4,\ldots,v_{2+x}\}$. Observe that $A$ has precisely $x$ vertices. Let us consider the 
presence or absence of the $2x$ edges in $sA \cup tA$ 
after each of the edges in $H_{n,m}$ fail with identical probability $\rho$. The probability that precisely $i$ edges in $sA$ do not fail equals $\binom{x}{i}(1-\rho)^i\rho^{x-i}$, and if those edges are $sB$ for some vertex set $B\in \binom{A}{i}$, then the probability that each of the edges in $tB$ fail equals $\rho^i$. 
Therefore,
\begin{equation*}
U_{m-2}(\rho)=\sum_{i=0}^{x}\binom{x}{i}(1-\rho)^i\rho^{x-i}\rho^iU_{m-2-x}(\rho)
\leq \rho^x \sum_{i=0}^{x}\binom{x}{i}(1-\rho)^iU_{m-2-x-i}(\rho). \qedhere
\end{equation*}   

\end{enumerate}
\end{proof}

\begin{theorem}\label{theorem:3-UMRTTGs}
For each pair of integer $n$ and $m$ such that $n\geq 6$ and $5\leq m\leq 2n-3$, the only $3$-UMRTTG in $T_{n,m}$ is $H_{n,m}$. 
\end{theorem}
\begin{proof}
Let $n$ be a fixed integer such that $n\geq 6$. We will employ strong induction over $m\in \{5,6,\ldots,2n-3\}$ to prove the statement. Fix $m$ in $\{5,6,\ldots,2n-3\}$ and assume that $H_{n,m'}$ is the only $3$-UMRTTG in $T_{n,m'}$ for each $m'$ in $\{5,6,\ldots,2n-3\}$ such that $m'<m$. We will prove that $H_{n,m}$ is the only $3$-UMRTTG in $T_{n,m}$. If $m\in \{5,6,7,8\}$ then $H_{n,m}$ the only two-terminal  graph in $T_{n,m}^*$ thus the result follows from Lemma~\ref{lemma:conditions}. If $m=9$ then $T_{n,m}^*=\{K_5-e,H_{n,9}\}$ for some edge $e$ in $K_5$ whose endpoints are not terminals, and if $m=10$ then $T_{n,m}^*=\{K_5,H_{n,10}\}$. It is a simple exercise to prove that $U_9(\rho)<U_{K_5-e}^3(\rho)$ and $U_{10}(\rho)<U_{K_5}^3(\rho)$ 
for every $\rho \in (0,1)$. If $m\geq 11$ then, by Lemma~\ref{lemma:conditions}, it is enough to prove that $U_{m}(\rho)< U_G^3(\rho)$ for each $G$ in $T_{n,m}^*$ and for every $\rho$ in $(0,1)$. 
Among all nonisolated vertices in $G$, let us choose one vertex $v$ with minimum degree in $\hat{G}$. 
If $deg_{\hat{G}}(v)=0$ then the only neighbors of $v$ in $G$ are $s$ and $t$. In this case we can see that the $2$-path consisting of the edges $sv$ and $vt$ either fails or not, so $U_G^3(\rho)=(1-(1-\rho)^2)U_{G-sv-vt}^3(\rho)$, and by the inductive assumption $U_{G-sv-vt}^3(\rho)>U_{m-2}(\rho)$ for every $\rho$ in $(0,1)$. In this case, $U_{G}^3(\rho)=(1-(1-\rho)^2)U_{G-sv-vt}^3(\rho)>(1-(1-\rho)^2)U_{m-2}(\rho)=U_m(\rho)$ for every $\rho$ in $(0,1)$, where we used Lemma~\ref{lemma:Um}\ref{it1} in the last equality.

Now, if $deg_{\hat{G}}(v)\geq 1$ then we define $X=N_{\hat{G}}(v)$ and $x=|X|$. If both edges $sv$ and $vt$ do not fail (event with probability $(1-\rho)^2$) then there exists a $2$-path joining $s$ and $t$ in $G$ and the $3$-constrained two-terminal unreliability of the resulting subgraph is equal to $0$. If both edges $sv$ and $vt$ fail (event with probability $\rho^2$) then the $3$-constrained two-terminal unreliability of the resulting subgraph is identical to that of $G-v$.  
Otherwise, precisely one of the edges $sv$ or $vt$ does not fail (event with probability $2\rho(1-\rho)$). Without loss of generality, assume that $sv$ is the edge that does not fail. In this case, let us condition on the presence or absence of the edges belonging to $vX$. 
Let $i\in \{0,1,\ldots,x\}$. For each $A \in \binom{X}{i}$, if all edges in $sA$ do not fail and all edges in $s(X-A)$ fail (event with probability $(1-\rho)^i\rho^{x-i}$), then the $3$-constrained two-terminal unreliability of the resulting subgraph equals the product of the probability that all edges in $tA$ fail (event with probability $\rho^i$) and the $3$-constrained two-terminal unreliability of $G-v-tA$. Therefore,
\begin{equation*}
U_G^3(\rho) = 0(1-\rho)^2 + \rho^2U_{G-v}^3(\rho)+
2\rho(1-\rho)\sum_{i=0}^{x}\sum_{A \in \binom{X}{i}}(1-\rho)^i \rho^{x-i}\rho^iU_{G-v-tA}^3(\rho).
\end{equation*}
As $G-v$ has precisely $m-2-x$ edges and $|A|=i$, the graph $G-v-tA$ has precisely $m-2-x-i$ edges. In the following, we will use our inductive assumption and Lemma~\ref{lemma:Um}\ref{it3}. First, we need to prove that $x\leq \frac{m-7}{2}$. Let $\ell$ be the number of common neighbors between $s$ and $t$. As $m\geq 11 > \binom{5}{2}$ and $G$ is in $T_{n,m}^{*}$, the connected component including $s$ has at least $6$ vertices thus $\ell \geq 4$. Consequently, $deg_G(s)=deg_G(t)=\ell+1\geq 5$. As $x+2$ equals the minimum degree among all nonisolated nonterminal vertices in $G$, there are at least $\ell$ vertices in $G$ whose degree is at least $x+2$. The handshaking lemma applied to $G$ yields 
\begin{equation*}
2m = deg_G(s)+deg_G(t)+\sum_{i=3}^{n}deg_G(v_i)\geq 2(\ell+1)+\ell(x+2)\geq 4x+18,    
\end{equation*}
where the last step holds since $\ell\geq 4$. Dividing by $2$ we obtain that $m\geq 2x+9$, which implies that $x\leq \frac{m-9}{2} \leq \frac{m-7}{2}$, as required. 
Finally, using the inductive assumption we obtain that 
\begin{align*}
U_G^3(\rho) &= \rho^2U_{G-v}^3(\rho)+2\rho(1-\rho)\sum_{i=0}^{x}\sum_{A \in \binom{X}{i}}(1-\rho)^i\rho^{x-i}\rho^iU_{G-v-At}^3(p)\\
         &\geq \rho^2U_{m-x-2}(\rho)+2(1-\rho)\rho^{x+1}\sum_{i=0}^{x}\binom{x}{i}(1-\rho)^iU_{m-x-i-2} \geq U_m(\rho),
\end{align*}
where the last inequality holds by Lemma~\ref{lemma:Um}\ref{it3}. The theorem follows.
\end{proof}

The decision of existence or nonexistence of 3-UMRTTGs in the cases not covered by Theorem~\ref{theorem:3-UMRTTGs} is still unexplored. In the remaining of this section we give properties shared by all $d$-LMRTTGs near $0$ when $d\geq 3$,  $n\geq 6$ and $3n-5 \leq m \leq \binom{n}{2}$. The analysis is analogous to the determination of LMRTTGs near $0$ in $T_{n,m}$ given by Xie et al. in~\cite{2021Xie}.

Let $G$ be in $T_{n,m}$. For each $i\in \{1,2,\ldots,m\}$ and each positive integer $d$ we define $B_i^d(G)$ as $N_{m-i}^d(G)$. 
Let $\lambda_{st}^d(G)$ be the least positive integer $j$ 
such that $B_j^d(G)$ is positive. A \emph{trivial $d$-cutset} 
is an edge set $U$ of $G$ such that $G-U$ has no $d$-pathset and further 
$U$ includes all edges incident to $s$ or to $t$.

\begin{lemma}\label{lemma:minimize:n-2}
Let $d$, $n$ and $m$ be integers such that $d\geq 3$, $n\geq 6$, and $3n-5 \leq m\leq \binom{n}{2}$. For each $G$ in $T_{n,m}$ it holds that   $\lambda_{st}^d(G)\leq n-1$, and the equality is attained if and only if 
$G$ is in $T_{n,m}^u$. 
\end{lemma}
\begin{proof}
Let $n$, $m$ and $d$ be as in the statement, and let $G$ be any two-terminal graph in $T_{n,m}$. If $G$ is in $T_{n,m}^u$ then 
$G$ has $st$ as well as precisely $n-2$ edge-disjoint $2$-paths joining $s$ and $t$. By Menger theorem, $\lambda(G) = n-1$. 
As each of the previous $n-1$ edge-disjoint paths have length $1$ or $2$ and $d\geq 3$, it follows that $\lambda_{st}^d(G)=n-1$.  
If $G$ does not belong to $T_{n,m}^u$ then $G$ has a trivial $d$-cutset 
composed by at most $n-2$ edges thus $B_{n-2}^d(G)>0$ and $\lambda_{st}^d(G)\leq n-2$. 
The lemma follows.
\end{proof}

\begin{lemma}\label{lemma:minimize:n-2+k}
Let $d$, $n$ and $m$ be integers such that $d\geq 3$, $n\geq 6$, and $3n-5\leq m\leq \binom{n}{2}$. Let $G$ be in $T_{n,m}^u$. 
Then, for each 
$i\in \{1,\ldots,\lambda(\hat{G})\}$ it holds that $B_{n-2+i}^d(G)= 2\binom{m-n+1}{i-1}$ and $B_{n-1+\lambda(\hat{G})}>2\binom{m-n+1}{\lambda(\hat{G})}$. 
\end{lemma}
\begin{proof}
Let $d$, $n$ and $m$ be as in the statement and let $G$ be in $T_{n,m}^u$. Let $i\in \{1,2,\ldots,\lambda(\hat{G})\}$. Each trivial $d$-cutset in $G$ composed by $n-2+i$ edges consists of all the $n-1$ incident edges of $s$ or $t$ plus $i-1$ extra edges, so $B_{n-2+i}^d(G)\geq 2\binom{m-n+1}{i-1}$. We will prove that $G$ has no nontrivial $d$-cutsets of cardinality $n-2+i$. 
In fact, a nontrivial $d$-cutset must remove $st$ as well as some nonempty edge set $sA$ where $A$ is strictly included in $V(\hat{G})$, plus all edges in $tA^c$ (otherwise we would have some $2$-path $sx,xt$ for some $x\in A^c$) and some edge set $U$ in $\hat{G}$. Specifically, if there exists some nontrivial $d$-cutset $U'$ in $G$ of cardinality $n-2+i$ then $U'=\{st\} \cup sA \cup tA^c \cup U$ for some edge set $U$ of $\hat{G}$ composed by precisely $i-1$ edges. As $|U|=i-1$ and $i-1 < \lambda(\hat{G})$, it follows that $\hat{G}-U$ is connected. Consequently, there exists some edge $xy$ in $G-U'$ such that $x\in A$ and $y\in A^c$, so $sy,yx,xt$ is a $3$-path included in $G-U'$. This means that $G-U'$ is not a $d$-cutset. Therefore, $G$ has no nontrivial $d$-cutsets, and $B_{n-2+i}^d(G)= 2\binom{m-n+1}{i-1}$.

Finally, if $i=\lambda(\hat{G})+1$ then there exists some edge set $U'$ of $\hat{G}$ composed by $i-1$ edges 
such that $\hat{G}-U'$ is not connected. Let $G_1$ be some connected component of $\hat{G}-U'$, $A=V(G_1)$ 
and set $U=\{st\} \cup sA \cup tA^c \cup U'$. It is clear that $U$ is a nontrivial $d$-cutset with $n-1+\lambda(\hat{G})$ edges thus $B_{n-1+\lambda(\hat{G})}>2\binom{m-n+1}{\lambda(\hat{G})}$. The lemma follows. 
\end{proof}

We close this section with a statement that gives a necessary condition for a two-terminal graph in $T_{n,m}$ to be a $d$-LMRTTG near $0$ 
when $d\geq 3$, $n\geq 6$ and $3n-5 \leq m \leq \binom{n}{2}$. 

\begin{proposition}\label{proposition:regularity}
Let $d$, $n$ and $m$ be integers such that $d\geq 3$, $n\geq 6$, and $3n-5 \leq m \leq \binom{n}{2}$. 
If $G$ is a $d$-LMRTTG near $0$ then $G$ is in $T_{n,m}^u$ and $\hat{G}$ is almost-regular.
\end{proposition}
\begin{proof}
Let $G$ be a $d$-LMRTTG near $0$ in $T_{n,m}$. 
By Remark~\ref{remark:local}\ref{it2:local} we know that $\lambda_{st}^d(G)$ must be maximum. By Lemma~\ref{lemma:minimize:n-2} we know that $\lambda_{st}^d(G)\leq n-1$ and 
the equality is met if and only if $G$ is in $T_{n,m}^u$. 
As $G$ is in $T_{n,m}^u$ and $G$ is a $d$-LMRTTG near $0$, Lemma~\ref{lemma:minimize:n-2+k} gives that $\lambda(\hat{G})$ must be maximum. It is known~\cite{1962Harary} that each simple graph with maximum connectivity is almost-regular. The proposition follows.
\end{proof}

\section{Nonexistence of d-UMRTTGs when d is greater than 3}\label{section:contrast}
We will show in this section that the nonexistence results stated in 
Section~\ref{section:background} are inherited when the distance constraint $d$ is greater than $3$. The key is Lemma~\ref{lemma:inherit} and Lemma~\ref{lemma:distances}, whose proofs are elementary.
\begin{lemma}\label{lemma:inherit}
Let $d$ be any positive integer and let $G$ be any two-terminal graph in $T_{n,m}$ such that the length of each of its paths joining its terminals is at most $d$. Then, $R_G^d(\rho)=R_G(\rho)$ for every $\rho$ in $[0,1]$. 
\end{lemma}
\begin{proof}
Let $G$ be any two-terminal graph satisfying the conditions of the statement, and let $i$ be any integer in $\{1,2,\ldots,m\}$. It is enough to prove that $N_i^d(G)=N_i(G)$. By definition we know that 
$N_i^d(G)\leq N_i(G)$. Consider any spanning subgraph $H$ of $G$ composed by $i$ edges such that there exists some path in $H$ joining the terminals of $G$. By our assumption, the length of such a path is at most $d$ thus $H$ is also a $d$-pathset of $G$ and $N_i^d(G)\geq N_i(G)$. Therefore, $N_i^d(G)=N_i(G)$, and the lemma follows. 
\end{proof}

\begin{lemma}\label{lemma:distances}
Let $G$ be any two-terminal graph in $T_{n,m}$ and let $d$ be any positive integer. Then, $N_i^d(G)=N_i(G)$ for each $i\in \{1,2,\ldots,d\}$.    
\end{lemma}
\begin{proof}
Consider a two-terminal graph $G$ in $T_{n,m}$ and positive integers $d$ and $i$ such that $i\in \{1,2,\ldots,d\}$. 
By definition, we know that $N_i^d(G)\leq N_i(G)$. Let $H$ by any spanning subgraph of $G$ composed by $i$ edges 
joining its terminals. Clearly, each path in $H$ joining its terminals has at most $|E(H)|$ edges that is precisely $i$ edges thus $N_i^d(G)\geq N_i(G)$. 
Therefore, $N_i^d(G)=N_i(G)$, and the lemma follows.
\end{proof}

\begin{lemma}\label{lemma:d4}
Let $d$, $n$ and $m$ be integers such that $d\geq 4$, $n\geq 5$, and $m\geq 5$. All the following assertions hold:
\begin{enumerate}[label=(\roman*)]
    \item\label{it1:low} If $m\leq 2n-3$ then $H_{n,m}$ is the only $d$-LMRTTG near $1$ in $T_{n,m}$.
    \item\label{it2:medium} If $2n-3<m\leq 3n-6$ then $A_{n,m-2n+3}$ is the only $d$-LMRTTG near $1$ in $T_{n,m}$.
    \item\label{it3:high} If $3n-5\leq m\leq \binom{n}{2}-2$ and $G$ is a $d$-LMRTTG near $1$ in $T_{n,m}$ then $\hat{G}$ is not almost-regular. 
\end{enumerate}
\end{lemma}
\begin{proof}
Let us proof each of the assertions separately.
\begin{enumerate}[label=(\roman*)]
    \item Let $G$ be any two-terminal graph in $T_{n,m}$ that is not isomorphic to $H_{n,m}$. 
    By Lemma~\ref{lemma:LMRTTG1} we know that $H_{n,m}$ is the only LMRTTG near $1$ in $T_{n,m}$. By Remark~\ref{remark:uniqueness}, there exists $\delta>0$ such that for every $\rho$ in $(1-\delta,1)$ it holds that $R_{H_{n,m}}(\rho)>R_G(\rho)$. Observe that the length of each path joining $s$ and $t$ in $H_{n,m}$ is at most $3$. As $d\geq 4$, Lemma~\ref{lemma:inherit} gives that $R_{H_{n,m}}^d(\rho)=R_{H_{n,m}}(\rho)$ for every $\rho$ in $[0,1]$. Clearly, $R_G(\rho)\geq R_G^d(\rho)$ for every $\rho$ in $[0,1]$. Then, $R_{H_{n,m}}^d(\rho)=R_{H_{n,m}}(\rho)>R_G(\rho)\geq R_G^d(\rho)$ for every $\rho$ in $(1-\delta,1)$, as required.  
    \item The length of each path joining $s$ and $t$ in $A_{n,m-2n+3}$ is at most $4$. The proof is analogous to~\ref{it1:low}. 
    \item Let $G$ be any two-terminal graph in $T_{n,m}$ and let $i$ be any integer in $\{1,2,3,4\}$. As $d\geq 4$, Lemma~\ref{lemma:distances} gives that $N_i^d(G)=N_i(G)$. As a consequence, $T_{n,m}^4(i)=T_{n,m}(i)$ for each $i\in \{1,2,3,4\}$. During the proof of Theorem~\ref{theorem:Xie}\ref{teo-not-regular}, Xie et al. proved that each $G$ in $T_{n,m}(4)$ is in $T_{n,m}^u$ and $\hat{G}$ is not almost-regular. Consequently, the set $T_{n,m}^4(4)$ consists of two-terminal graphs $G$ in $T_{n,m}^u$ such that $\hat{G}$ is not almost-regular. As $T_{n,m}^4(m)$ is included in $T_{n,m}^4(4)$, the result follows from Lemma~\ref{lemma:existenceLMRTTG}. \qedhere
\end{enumerate}
\end{proof}

\begin{theorem}
Let $d$ and $n$ be any pair of integers such that $d\geq 4$ and $n\geq 11$. There is no $d$-UMRTTG in $T_{n,m}$ when $20\leq m \leq 3n-9$ or when $3n-5\leq m \leq \binom{n}{2}-2$.
\end{theorem}
\begin{proof}
Let $d$ and $n$ be any pair of integers such that $d\geq 4$ and $n\geq 11$. 
First, let us consider the case in which $20 \leq m \leq 3n-9$. 
Consider the two-terminal graph $G_{n,m}$ defined in Lemma~\ref{lemma:half}. 
If $m\leq 2n-3$ then, by Lemma~\ref{lemma:d4}\ref{it1:low}, we know that $H_{n,m}$ is the only $d$-LMRTTG near $1$ in $T_{n,m}$. By Lemma~\ref{lemma:half}, we know that $R_{G_{n,m}}(1/2)>R_{H_{n,m}}(1/2)$. As both $G_{n,m}$ and $H_{n,m}$ satisfy the conditions of Lemma~\ref{lemma:inherit} we conclude that $R_{G_{n,m}}^d(1/2)>R_{H_{n,m}}^d(1/2)$ thus by Remark~\ref{remark:uniqueness} 
no $d$-UMRTTG exists. Now, if $2n-3<m\leq 3n-9$ then the reasoning is analogous. Lemma~\ref{lemma:d4}\ref{it2:medium} gives that $A_{n,m-2n+3}$ 
is the only $d$-LMRTTG near $1$ in $T_{n,m}$, while Lemma~\ref{lemma:half} 
confirms again that $R_{G_{n,m}}^d(1/2)>R_{A_{n,m-2n+3}}^d(1/2)$ thus by Remark~\ref{remark:uniqueness} no $d$-UMRTTG exists.

Finally, if $3n-5\leq m \leq \binom{n}{2}-2$ then Lemma~\ref{lemma:d4}\ref{it3:high} gives that each $d$-LMRTTG near $1$ in $T_{n,m}$ satisfies that $\hat{G}$ is not almost-regular. However, Proposition~\ref{proposition:regularity} gives that $\hat{G}$ is almost regular. 
As each $d$-UMRTTG must be both $d$-LMRTTG near $0$ and $1$ no $d$-UMRTTG exists, and the theorem follows.
\end{proof}


\section*{Acknowledgment}
This work is partially supported by City University of New York project entitled \emph{On the problem of characterizing graphs with maximum number of spanning trees} with grant number 66165-00. The author wants to thank Professors Mart\'in Safe and Louis Petingi for their helpful comments that improved the presentation of this manuscript as well as Professor Guillermo Durán for his permanent support.

\end{document}